\documentclass[12pt,leqno]{article}
\usepackage{amssymb,graphics,latexsym}
\usepackage{graphicx}
\usepackage{amsmath,amsthm}
\usepackage{amsfonts,epsfig}

\theoremstyle{plain}
\newtheorem{theorem}{Theorem}[section]
\newtheorem{corollary}[theorem]{Corollary}
\newtheorem{lemma}[theorem]{Lemma}

\theoremstyle{remark}

\newcommand{\complex}{\mathbb{C}}
\newcommand{\nul}{{\rm nullity\, }}

\newcommand{\rank}{{\rm rank\, }}

\begin{document}



\bibliographystyle{plain}
\title{
Power Partial Isometry Index and Ascent of a Finite Matrix}

\author{
Hwa-Long Gau\thanks{Department of Mathematics,
National Central University, Chung-Li 32001, Taiwan, Republic of China
(hlgau@math.ncu.edu.tw).}
\and
Pei Yuan Wu\thanks{Department of Applied Mathematics, National Chiao Tung University,
Hsinchu 30010, Taiwan, Republic of China (pywu@math.nctu.edu.tw).}
}

\pagestyle{myheadings}
\markboth{Hwa-Long Gau and Pei Yuan Wu}{Power Partial Isometry Index and Ascent of a Finite Matrix}
\maketitle

\begin{abstract}
We give a complete characterization of nonnegative integers $j$ and $k$ and a positive integer $n$ for which there is an $n$-by-$n$ matrix with its power partial isometry index equal to $j$ and its ascent equal to $k$. Recall that the power partial isometry index $p(A)$ of a matrix $A$ is the supremum, possibly infinity, of nonnegative integers $j$ such that $I, A, A^2, \ldots, A^j$ are all partial isometries while the ascent $a(A)$ of $A$ is the smallest integer $k\ge 0$ for which $\ker A^k$ equals $\ker A^{k+1}$. It was known before that, for any matrix $A$, either $p(A)\le\min\{a(A), n-1\}$ or $p(A)=\infty$. In this paper, we prove more precisely that there is an $n$-by-$n$ matrix $A$ such that $p(A)=j$ and $a(A)=k$ if and only if one of the following conditions holds: (a) $j=k\le n-1$, (b) $j\le k-1$ and $j+k\le n-1$, and (c) $j\le k-2$ and $j+k=n$. This answers a question we asked in a previous paper.
\end{abstract}



\section{Introduction} \label{s1}
Let $A$ be an $n$-by-$n$ complex matrix. The \emph{power partial isometry index} $p(A)$ of $A$ is, by definition, the supremum of the nonnegative integers $j$ for which $I, A, A^2, \ldots, A^j$ are all partial isometries. Recall that $A$ is a \emph{partial isometry} if $\|Ax\|=\|x\|$ for all vectors $x$ of $\complex^n$ which are in the orthogonal complement $(\ker A)^{\perp}$ of $\ker A$. The \emph{ascent} $a(A)$ of $A$ is the smallest nonnegative integer $k$ for which $\ker A^k=\ker A^{k+1}$. The relation between these two parameters of $A$ was first explored in \cite{1}. In particular, it was shown in \cite[Corollary 2.5]{1} that $0\le p(A)\le\min\{a(A), n-1\}$ or $p(A)=\infty$. We asked in \cite[Question 3.7]{1} that whether such conditions on $p(A)$ and $a(A)$ guarantee their attainment by some $n$-by-$n$ matrix $A$. In this paper, we show that this is not always the case. It turns out that the situation is more delicate than what we have expected. More precisely, we prove that, for nonnegative integers $j$ and $k$ and a positive integer $n$, there is an $n$-by-$n$ matrix $A$ with $p(A)=j$ and $a(A)=k$ if and only if one of the following three conditions holds: (a) $j=k\le n-1$, (b) $j\le k-1$ and $j+k\le n-1$, and (c) $j\le k-2$ and $j+k=n$. This settles \cite[Question 3.7]{1} completely. The proof of it depends on the special matrix representation, under unitary similarity, of a matrix $A$ for which $A, A^2, \ldots, A^j$ are all partial isometries for a certain $j$, $1\le j\le\infty$ (cf. \cite[Theorems 2.2 and 2.4]{1}). We will review the necessary ingredients from \cite{1} in Section \ref{s2} below. Section \ref{s3} then gives the proof of our main result.

Partial isometries were first studied in \cite{3} and their properties have since been summaried in \cite[Chapter 15]{2}. Power partial isometries were considered first in \cite{4}.

\vspace{5mm}

\section{Preliminaries}\label{s2}
We start with the following result from \cite[Theorem 2.2]{1}

\begin{theorem} \label{t21}
Let $A$ be an $n$-by-$n$ matrix and $1\le j\le a(A)$. Then $A, A^2, \ldots, A^j$ are partial isometries if and only if $A$ is unitarily similar to a matrix of the form
\begin{equation}\label{e1}
A'=\left[\begin{array}{ccccc} 0 & A_1 & & & \\ & 0 & \ddots & & \\ & & \ddots & A_{j-1} & \\ & & & 0 & B\\ & & & & C\end{array}\right] \ on \  \complex^n=\complex^{n_1}\oplus\cdots\oplus\complex^{n_j}\oplus\complex^{m},
\end{equation}
where the $A_{\ell}$'s satisfy $A_{\ell}^*A_{\ell}=I_{n_{\ell+1}}$ for $1\le \ell\le j-1$, and $B$ and $C$ satisfy $B^*B+C^*C=I_m$. In this case, $n_{\ell}=\nul A$ if $\ell=1$, and $\nul A^{\ell}-\nul A^{\ell-1}$ if $2\le \ell\le j$, and $m=\rank A^j$.
\end{theorem}

Here, for any $p\ge 1$, $I_p$ denotes the $p$-by-$p$ identity matrix, and, for any matrix $B$, $\nul B$ means $\dim\ker B$.

A consequence of Theorem \ref{t21} is the next result from \cite[Theorem 2.4]{1}.

\begin{theorem} \label{t22}
Let $A$ be an $n$-by-$n$ matrix and $j>a(A)$. Then the following conditions are equivalent:
\begin{enumerate}
\item[\rm (a)] $A, A^2, \ldots, A^j$ are partial isometries,
\item[\rm (b)] $A$ is unitarily similar to a matrix of the form $U\oplus J_{k_1}\oplus\cdots\oplus J_{k_m}$, where $U$ is unitary and $a(A)=k_1\ge\cdots\ge k_m\ge 1$, and
\item[\rm (c)] $A^{\ell}$ is a partial isometry for all $\ell\ge 1$.
\end{enumerate}
\end{theorem}

Here $J_q$ denotes the $q$-by-$q$ \emph{Jordan block}
$$\left[
    \begin{array}{cccc}
      0 & 1 &   &   \\
        & 0 & \ddots &   \\
        &   &  \ddots & 1 \\
        &   &   & 0
    \end{array}
  \right].$$

An easy corollary of the preceding theorem is the following estimate for $p(A)$ from \cite[Corollary 2.5]{1}.

\begin{corollary}\label{c23}
If $A$ is an $n$-by-$n$ matrix, then $0\le p(A)\le\min\{a(A), n-1\}$ or $p(A)=\infty$.
\end{corollary}

In constructing the examples for our main result, we need the class of $S_n$-matrices. Recall that an $n$-by-$n$ matrix $A$ is said to be of {\em class} $S_n$ if $A$ is a contraction ($\|A\|\equiv\max\{\|Ax\|: x\in\mathbb{C}^n, \|x\|=1\}\le 1$), its eigenvalues all have moduli strictly less than 1, and $\rank(I_n-A^*A)=1$. Such matrices are finite-dimensional versions of the \emph{compressions of the shift} $S(\phi)$ studied first by Sarason \cite{5}, which later featured prominently in the Sz.-Nagy--Foia\c{s} contraction theory \cite{6}. A special example of $S_n$-matrices is the Jordan block $J_n$. In fact, many properties of $J_n$ can be extended to those for the more general $S_n$-matrices. Part (a) of the following theorem from \cite[Proposition 3.1]{1} is one such instance.

\begin{theorem} \label{t24}
Let $A$ be a noninvertible $S_n$-matrix. Then
\begin{enumerate}
\item[\rm (a)] $a(A)$ equals the algebraic multiplicity of the eigenvalue $0$ of $A$,
\item[\rm (b)] $p(A)=a(A)$ or $\infty$, and
\item[\rm (c)] $p(A)=\infty$ if and only if $A$ is unitarily similar to $J_n$.
\end{enumerate}
\end{theorem}

\vspace{5mm}

\section{Main Result}\label{s3}
The following is the main theorem of this paper.

\begin{theorem} \label{t31}
Let $j$ and $k$ be nonnegative integers and $n$ be a positive integer. Then there is an $n$-by-$n$ matrix $A$ such that $p(A)=j$ and $a(A)=k$ if and only if one of the following conditions holds:
\begin{enumerate}
\item[\rm (a)] $j=k\le n-1$,
\item[\rm (b)] $j\le k-1$ and $j+k\le n-1$, and
\item[\rm (c)] $j\le k-2$ and $j+k=n$.
\end{enumerate}
\end{theorem}

To prove this, we need the next two lemmas.

\begin{lemma}\label{l32}
If $A$ is an $n$-by-$n$ matrix, which is unitarily similar to a matrix $A'$ as in $(\ref{e1})$ with $1\le j\le a(A)$, then $(\mbox{\rm a})$ $p(A)=j+p(C)$, and $(\mbox{\rm b})$ $a(A)=j+a(C)$.
\end{lemma}

\begin{proof}
For any $\ell\ge 0$, multiplying $A'$ with itself $j+\ell$ times results in
\begin{equation}\label{e2}
A'^{j+\ell}=\left[\begin{array}{cccc}
0 & \cdots & 0 &   (\prod_{p=1}^{j-1}A_p)BC^{\ell}\\
0 & \cdots & 0 &   (\prod_{p=2}^{j-1}A_p)BC^{\ell+1}\\
\vdots &   & \vdots   & \vdots\\
0 & \cdots & 0 &   BC^{j+\ell-1}\\
0 & \cdots & 0 &   C^{j+\ell}
\end{array}\right].
\end{equation}

(a) Note that $A'^{j+\ell}$ is a partial isometry if and only if ${A'^{j+\ell}}^*A'^{j+\ell}$ is an (orthogonal) projection (cf. \cite[Problem 127]{2}), and the latter is equivalent to
\begin{equation}\label{e3}
(\sum_{q=\ell}^{j+\ell-1}{C^q}^*B^*(\prod_{p=q-\ell+1}^{j-1}A_p)^*(\prod_{p=q-\ell+1}^{j-1}A_p)BC^q)+{C^{j+\ell}}^*C^{j+\ell}
\end{equation}
being a projection. Making use of $A_p^*A_p=I_{n_{p+1}}$, $1\le p\le j-1$, and $B^*B+C^*C=I_m$, we can simplify (\ref{e3}) to
\begin{eqnarray*}
&& (\sum_{q=\ell}^{j+\ell-1}{C^q}^*B^*BC^q)+{C^{j+\ell}}^*C^{j+\ell}\\
&=& (\sum_{q=\ell}^{j+\ell-2}{C^q}^*B^*BC^q)+{C^{j+\ell-1}}^*(B^*B+C^*C)C^{j+\ell-1}\\
&=& (\sum_{q=\ell}^{j+\ell-2}{C^q}^*B^*BC^q)+{C^{j+\ell-1}}^*C^{j+\ell-1}\\
&=& (\sum_{q=\ell}^{j+\ell-3}{C^q}^*B^*BC^q)+{C^{j+\ell-2}}^*(B^*B+C^*C)C^{j+\ell-2}\\
&=& \cdots\\
&=& {C^{\ell}}^*C^{\ell}.
\end{eqnarray*}
Thus ${C^{\ell}}^*C^{\ell}$ is a projection, which is equivalent to $C^{\ell}$ being a partial isometry. From these, we conclude that $p(A)=p(A')=j+p(C)$.

(b) For any $\ell\ge 0$, let $s_{\ell}$ (resp., $t_{\ell}$) denote the geometric (resp., algebraic) multiplicity of the eigenvalue 0 of $A^{j+\ell}$, and let $u_{\ell}$ (resp., $v_{\ell}$) be the corresponding multiplicities of 0 of $C^{\ell}$. Obviously, we have $t_{\ell}=t_0$ for all $\ell\ge 0$ and $v_{\ell}=v_1$ for $\ell\ge 1$. We claim that $s_{\ell}=(\sum_{i=1}^j n_i)+u_{\ell}$ for $\ell\ge 0$. Indeed, let $x_1\oplus\cdots\oplus x_j\oplus y$ in $\complex^{n_1}\oplus\cdots\oplus\complex^{n_j}\oplus\complex^m$ be any vector in $\ker A'^{j+\ell}$. From (\ref{e2}), we have $(\prod_{p=q-\ell+1}^{j-1}A_p)BC^qy=0$, $\ell\le q\le j+\ell-1$, and $C^{j+\ell}y=0$. Since $A_p^*A_p=I_{n_{p+1}}$ for $1\le p\le j-1$, we obtain $BC^qy=0$ for $\ell\le q\le j+\ell-1$. Applying $B^*B+C^*C=I_m$ to the vector $C^{j+\ell-1}y$ yields that
$$C^{j+\ell-1}y=B^*(BC^{j+\ell-1}y)+C^*(CC^{j+\ell-1}y)=0+0=0.$$
We may then apply $B^*B+C^*C=I_m$ again to $C^{j+\ell-2}y$ as above to obtain $C^{j+\ell-2}y=0$. Repeating this process inductively, we finally reach $C^{\ell}y=0$, that is, $y$ is in $\ker C^{\ell}$. This shows that $\ker A'^{j+\ell}$ is contained in the subspace $\complex^{n_1}\oplus\cdots\oplus\complex^{n_j}\oplus\ker C^{\ell}$. Since the reversed containment is easily seen to be true, it follows that
$$s_{\ell}=\nul A^{j+\ell}=\nul A'^{j+\ell}=(\sum_{i=1}^j n_i)+u_{\ell}$$
for any $\ell\ge 0$ as claimed. Note that, for any matrix $T$, its ascent is equal to the smallest nonnegative integer $k$ for which the geometric and algebraic multiplicities of the eigenvalue 0 of $T^k$ coincide. Thus
$$u_{a(C)}=v_{a(C)}=\left\{\begin{array}{ll} v_1 \  & \mbox{if } \, a(C)\ge 1,\\ 0 & \mbox{if } \, a(C)=0,\end{array}\right.  \ \mbox{ and } \ u_{a(C)-1}<u_{a(C)}=v_1 \  \mbox{ if } \, a(C)\ge 1.$$
Therefore,
$$s_{a(C)}=\Big(\sum_{i=1}^j n_i\Big)+u_{a(C)}=\left\{\begin{array}{ll} \big(\sum\limits_{i=1}^j n_i\big)+v_1=t_0=t_{a(C)} \  & \mbox{if } \, a(C)\ge 1,\\ \sum\limits_{i=1}^j n_i & \mbox{if } \, a(C)=0,\end{array}\right.$$
where the third equality follows from the upper-triangular block structure of $A'$, and
$$s_{a(C)-1}=\Big(\sum_{i=1}^j n_i\Big)+u_{a(C)-1}<\Big(\sum\limits_{i=1}^j n_i\Big)+v_1=t_0=t_{a(C)-1}  \ \  \mbox{ if } \, a(C)\ge 1.$$
This shows that $j+a(C)$ is the smallest integer $k$ for which the geometric and algebraic multiplicities of the eigenvalue 0 of $A^k$ are equal to each other. Thus $a(A)=j+a(C)$ follows.
\end{proof}

\begin{lemma}\label{l33}
Let $A$ be an $n$-by-$n$ matrix with $p(A)<\infty$.
\begin{enumerate}
\item[\rm (a)] If $p(A)+a(A)>n$, then $p(A)=a(A)$.
\item[\rm (b)] If $p(A)+a(A)=n$, then $p(A)=a(A)$ or $p(A)\le a(A)-2$.
\end{enumerate}
\end{lemma}

\begin{proof}
By Theorem \ref{t21}, $A$ is unitarily similar to a matrix $A'$ in (\ref{e1}) with $j=p(A)$. In particular, this implies that $n_1\ge n_2\ge\cdots\ge n_j\ge 1$ for if $n_j=\nul A^j-\nul A^{j-1}=0$, then we would have $\ker A^j=\ker A^{j-1}$, which yields the contradictory $p(A)\le a(A)\le j-1$ by Corollary \ref{c23}.

(a) Assuming $p(A)+a(A)>n$, we first check that $n_j=1$. Indeed, if otherwise $n_j\ge 2$, then $n_i\ge 2$ for all $i$, $1\le i\le j$. Making use of Lemma \ref{l32}, we have
\begin{eqnarray*}
&& n=\Big(\sum_{i=1}^j n_i\Big)+m\ge 2j+(a(C)-p(C))\\
&=& 2p(A)+(a(A)-p(A))=p(A)+a(A)>n,
\end{eqnarray*}
which is a contradiction. Thus $n_j=1$. This means that $B$ is a $1$-by-$m$ matrix. If $p(A)<a(A)$, then $p(C)<a(C)$ by Lemma \ref{l32} again. In particular, this says that $a(C)>0$ or 0 is an eigenvalue of $C$. After a unitary similarity, we may assume that $C$ is upper triangular with a zero first column. Let $C$ be partitioned as
$$\left[\begin{array}{cc} 0 & C_1\\ 0 & C_2\end{array}\right],$$
where $C_1$ (resp., $C_2$) is a $1$-by-$(m-1)$ (resp., $(m-1)$-by-$(m-1)$) matrix. We deduce from $B^*B+C^*C=I_m$ that $B=[e^{i\theta} \ 0 \ \ldots \ 0]$ for some real $\theta$ and $C_1^*C_1+C_2^*C_2=I_{m-1}$. Thus $A'$ is of the form
$$
\left[\begin{array}{cccccc} 0 & A_1 & & & & \\ & 0 & \ddots & & & \\ & & \ddots & A_{j-1} & & \\ & & & 0 & e^{i\theta} & 0 \ \ldots \ 0\\ & & & & 0 & C_1\\ & & & & & C_2\end{array}\right] \ \mbox{ on } \  \complex^n=\complex^{n_1}\oplus\cdots\oplus\complex^{n_j}\oplus\complex\oplus\complex^{m-1}.$$
Theorem \ref{t21} then leads to the contradictory $p(A)=p(A')>j=p(A)$. We conclude from Corollary \ref{c23} that $p(A)=a(A)$.

(b) Assume that $p(A)+a(A)=n$ and $p(A)=a(A)-1$. We consider two separate cases, both of which will lead to contradictions:

(i) $a(C)-p(C)\le m-1$. In this case, we proceed as in (a) to first prove that $n_j=1$. Indeed, if $n_j\ge 2$, then
\begin{eqnarray*}
&& n-1=\Big(\sum_{i=1}^j n_i\Big)+m-1\ge 2j+(a(C)-p(C))\\
&=& 2p(A)+(a(A)-p(A))=a(A)+p(A)=n,
\end{eqnarray*}
which is a contradiction. Hence $n_j=1$ and $B$ is a $1$-by-$m$ matrix. Then, since $p(A)<a(A)$, the second-half arguments in proving (a) yield that $p(A)=a(A)$, which contradicts our assumption of $p(A)=a(A)-1$.

(ii) $a(C)-p(C)=m$. Note that this can happen only when $a(C)=m$ and $p(C)=0$. Thus $m=a(C)-p(C)=a(A)-p(A)=1$ by Lemma \ref{l32} and our assumption. This shows that $C$ is a 1-by-1 matrix, say, $C=[c]$ with $a(C)=1$ and $p(C)=0$. The former condition $a(C)=1$ yields that $c=0$, which results in $p(C)=\infty$, contradicting the latter $p(C)=0$.

We conclude that $p(A)+a(A)=n$ implies $p(A)\neq a(A)-1$.
\end{proof}

Finally, we are ready to prove Theorem \ref{t31}.

{\em Proof of Theorem $\ref{t31}$}.
The existence of an $n$-by-$n$ matrix $A$ with $p(A)=j$ and $a(A)=k$ implies, by Corollary \ref{c23}, that $j\le\min\{k, n-1\}$. Lemma \ref{l33} then yields that one of (a), (b) and (c) must hold.

For the converse, assume that (a) holds. If $j=k=0$, then $A=(1/2)I_n$ will do. Otherwise, we have $1\le j=k\le n-1$. Let $A$ be a noninvertible $S_n$-matrix whose eigenvalue 0 has algebraic multiplicity $k$. Then Theorem \ref{t24} gives $p(A)=j=k=a(A)$.

Next assume that (b) holds. We consider three cases separately:

(i) If $j=0$ and $k=n-1$, then let $A=[1/2]\oplus J_{n-1}$. In this case, we have $p(A)=0=j$ and $a(A)=n-1=k$.

(ii) If $1\le j\le k=n-1-j$, then let $A=A_1\oplus J_k$, where $A_1$ is an $S_{n-k}$-matrix with the algebraic multiplicity of its eigenvalue 0 equal to $j$. Since $j=n-k-1$, we have, by Theorem \ref{t24},
$$p(A)=\min\{p(A_1), p(J_k)\}=\min\{j, \infty\}=j$$
and
$$a(A)=\max\{a(A_1), a(J_k)\}=\max\{j, k\}=k.$$

(iii) If $j+k\le n-2$, then let $A=A_1\oplus A_2$, where $A_1$ (resp., $A_2$) is an $S_{j+1}$-matrix (resp., $S_{n-j-1}$-matrix) with the algebraic multiplicity of its eigenvalue 0 equal to $j$ (resp., $k$). Since $j<j+1$ (resp., $k<n-j-1$), we have, by Theorem \ref{t24}, $p(A_1)=a(A_1)=j$ (resp., $p(A_2)=a(A_2)=k$). Thus
$$p(A)=\min\{p(A_1), p(A_2)\}=\min\{j, k\}=j$$
and
$$a(A)=\max\{a(A_1), a(A_2)\}=\max\{j, k\}=k.$$

Finally, if (c) holds, then two cases are to be considered:

(i) $j=k-2$. In this case, let
$$A=\left[\begin{array}{ccccc} 0 & I_2 & & & \\ & 0 & \ddots & & \\ & & \ddots & I_2 & \\ & & & 0 & B\\ & & & & C\end{array}\right] \ \ \mbox{on} \  \ \complex^n=\underbrace{\complex^{2}\oplus\cdots\oplus\complex^{2}}_j\oplus\complex^{2},$$
where $B={\scriptsize\left[\begin{array}{cc} 1 & 0\\ 0 & 1/\sqrt{2}\end{array}\right]}$ and $C={\scriptsize\left[\begin{array}{cc} 0 & 1/\sqrt{2}\\ 0 & 0\end{array}\right]}$. Since $n=j+k=j+(j+2)=2j+2$, $A$ is indeed an $n$-by-$n$ matrix with $B^*B+C^*C=I_2$. We infer from Lemma \ref{l32} (a) (resp., (b)) that
$$p(A)=j+p(C)=j+0=j \ \ \ (\mbox{resp.}, \  a(A)=j+a(C)=j+2=k).$$

(ii) $j\le k-3$. Let $m=k-j\ge 3$, and let
$$A=\left[\begin{array}{ccccc} 0 & I_2 & & & \\ & 0 & \ddots & & \\ & & \ddots & I_2 & \\ & & & 0 & B\\ & & & & C\end{array}\right] \ \ \mbox{on} \  \ \complex^n=\underbrace{\complex^{2}\oplus\cdots\oplus\complex^{2}}_j\oplus\complex^{m},$$
where
$$\begin{array}{l}\vspace*{-1.5mm}\ \ \ \ \hspace{21mm} \ \ \ \ \overbrace{\ \hspace{13mm} \ }^{\displaystyle m-3}\\
B=\left[\begin{array}{cccccc} 1 & 0 & 0 & \cdots &  0 & 0\\ 0 & 1/\sqrt{2} & 0 &  \cdots &  0 & 1/2\end{array}\right]\end{array}$$
and
$$C=\begin{array}{l}\vspace*{-1.5mm}\ \ \ \ \hspace{16mm} \ \ \ \ \overbrace{\ \hspace{21mm} \ }^{\displaystyle m-3} \\
\left[
       \begin{array}{ccccccc}
         0 & -1/\sqrt{2} & 0 & 0 & \cdots & 0 & 1/2 \\
           & 0  & 1 & 0 & \cdots & 0 & 0  \\
           &   & 0  & \ddots & \ddots  & \vdots  & \vdots  \\
           &   &   &  \ddots & \ddots & 0 & 0 \\
           &   &   & & \ddots  &  1 & 0   \\
           &   &   & &  & 0 &  1/\sqrt{2}  \\
           &   &   &   &   &  & 0
       \end{array}
     \right]. \end{array}$$
Since $n=j+k=2j+m$, $A$ is an $n$-by-$n$ matrix with $B^*B+C^*C=I_m$. Again, we infer from Lemma \ref{l32} (a) that
$$p(A)=j+p(C)=j+0=j,$$
where the second equality follows from the fact that $C^*C$ is not a projection and hence $C$ is not a partial isometry. On the other hand, Lemma \ref{l32} (b) implies that
$$a(A)=j+a(C)=j+m=j+(k-j)=k,$$
where the second equality holds because $C$ is similar to $J_m$. \hfill\qed

\bigskip
{\bf Acknowledgements.} The two authors acknowledge the supports from the National Science Council of the Republic of China under NSC-102-2115-M-008-007 and NSC-102-2115-M-009-007, respectively. The second author was also supported by the MOE-ATU project.


\end{document}